\documentclass[smallextended]{svjour3}       
\smartqed  
\usepackage{graphicx}
\begin{document}

\title{Computing multiple zeros by using a parameter in Newton-Secant method
}

\titlerunning{Computing multiple zeros by using a parameter in Newton-Secant method}        

\author{Massimiliano Ferrara \and Somayeh Sharifi \and Mehdi
Salimi}


\institute{M. Ferrara \at
              Department of Law and Economics, University Mediterranea of
Reggio Calabria, Italy\\
                            \email{massimiliano.ferrara@unirc.it}           
           \and
           S. Sharifi \at
              Young Researchers and Elite Club, Hamedan Branch, Islamic
Azad University, Hamedan, Iran\\
 \email{s.sharifi@iauh.ac.ir}
               \and
           M. Salimi \at
           Center for Dynamics, Department of Mathematics, Technische
Universit{\"a}t Dresden, Germany\\
              MEDAlics, Research Centre at
the University Dante Alighieri, Reggio Calabria, Italy\\
\\
 \email{mehdi.salimi@tu-dresden.de\\
 mehdi.salimi@medalics.org}}

\date{Received: date / Accepted: date}

\maketitle
\begin{abstract}
\noindent In this paper, we modify the Newton-Secant method with
third order of convergence for finding multiple roots of nonlinear
equations. Per iteration this method requires two evaluations of
the function and one evaluation of its first derivative. This
method has the efficiency index equal to $3^{\frac{1}{3}}\approx
1.44225$. We describe the analysis of the proposed method along
with numerical experiments including comparison with existing
methods. Moreover, the dynamics of the proposed method are shown
with some comparisons to the other existing methods.
\end{abstract}

\section{Introduction}
\label{intro} Solving nonlinear equations based on iterative
methods is a basic and extremely valuable tool in all fields in
the science as in economics, engineering and physics. The
important appearance related to these methods are order of
convergence and number of function evaluations. Therefore, it is
favorable to attain as high as possible convergence order with
fixed number of function evaluations each iteration. The aim of
this paper is to modify third order Newton-Secant method for
solving nonlinear equations for multiple zeros with same order of
convergence and efficiency index. In addition, the efficiency
index of an iterative method of order $p$ requiring $k$ function
evaluations per iteration is defined by $E(k,p) = \sqrt[k]{p}$,
see~\cite{Ostrowski}.

Let $\alpha$ is multi roots of $f(x)=0$ with multiplicity $m$ and
$x\neq \alpha$, then we can write $f(x)=(x-\alpha)^mq(x)$, where
$\lim_{x\rightarrow \alpha}q(x)\neq 0$. The function $f\in
C^m[a,b]$ has a zero of multiplicity $m$ at $\alpha$ in $]a,b[$ if
and only if
$0=f(\alpha)=f'(\alpha)=f''(\alpha)=\cdots=f^{(m-1)}(\alpha)$, but
$f^{(m)}(\alpha)\neq 0$. Newton method for finding a simple zero
\emph{$\alpha$} have been modified by Scheroder to find multiple
zero of a non-linear equation which is of the form
$x_{n+1}=x_n-m\frac{f(x_n)}{f'(x_n)}$ converges quadratically
\cite{scheroder}.

Moreover, there have been many attempts to construct methods for
finding multiple roots e.g.\ the work of Chun et al. \cite{chun},
Dong \cite{Dong}, Hansen and Patrick \cite{Halley}, Osada
\cite{Osada}, Victory and Neta \cite{victor} are proposed various
iterative methods for finding multiple zeros $\alpha$ of a
nonlinear equation $f(x)=0$ where multiplicity $m$ is known.

This paper is organized as follows: Section \ref{sec:2} is devoted
to the construction and convergence analysis of a new method with
convergence order three. Computational aspects, comparisons and
dynamic behavior with other methods are illustrated in Section
\ref{sec:3}. Finally, a conclusion is provided in Section
\ref{sec:4}.
\section{Description of the method and convergence analysis} \label{sec:2} In this section, we propose
a new modification of Newton-Secant's method to find multiple
zeros. Newton-Secant's method is

\begin{equation}
  \begin{array}{lrl}\label{c1}
 y_{n}=x_{n}-\frac{f(x_{n})}{f'(x_{n})},&\\
  [1ex]
x_{n+1}=x_n-\frac{f(x_n)}{f(x_n)-f(y_n)}\frac{f(x_n)}{f'(x_n)},\quad(n=0,1,\ldots),
      \end{array}
\end{equation}
that order of convergence is three for simple roots. We aim at
extending the method (\ref{c1}) for multiple roots and build a
method according to (\ref{c1}) without any additional evaluations
of the function or its derivatives by using of the parameter. In
other words, convergence order of Newton-Secant method for
approximating simple zero of nonlinear equations is three, while
convergence order of this method is linear for finding multiple
zeros. Therefore, we use a parameter $\theta$ for
solving this problem in the second term.\\
We have

\begin{equation}
  \begin{array}{lrl}\label{ccc2}
 y_{n}=x_{n}-\frac{f(x_{n})}{f'(x_{n})},&\\
  [1ex]
x_{n+1}=x_n-\frac{\theta f(x_n)}{\theta
f(x_n)-f(y_n)}\frac{f(x_n)}{f'(x_n)}.
      \end{array}
\end{equation}
The order of convergence of the preceding method will be analyzed
and the method will be adjusted accordingly to prove the following
expected theorem.
\begin{theorem}\label{cc}
Let $\alpha \in D$ be a multiple zero of a sufficiently
differentiable function $f:D \subset \mathbf{R} \rightarrow
\mathbf{R}$ for an open interval $D$ with the multiplicity $m$,
which includes $x_0$ as an initial approximation of $\alpha$.
Then, the method (\ref{ccc2}) has order three and
$\theta=\left(\frac{-1+m}{m}\right)^{-1+m}.$
\end{theorem}
\begin{proof}
Let $e_{n}:=x_{n}-\alpha$, $e_{n,y}:=y_{n}-\alpha$,
$c_{i}:=\frac{m!}{(m+i)!}\frac{f^{(m+i)}(\alpha)}{f^{(m)}(\alpha)}$.
Using the fact that $f(\alpha)=0$, Taylor expansion of $f$ at
$\alpha$ yields
\begin{equation}\label{a1}
f(x_{n}) =e_n^m\left(c_{0} +c_1e_n+
c_{2}e_{n}^{2}+c_{3}e_{n}^{3}\right)+O(e_n^{4}),
\end{equation}
and
\begin{equation}\label{a2}
f^{'}(x_{n})
=e_n^{m-1}(m+(m+1)c_1e_n+(m+2)c_2e_n^2+(m+3)c_3e_n^3+O(e_n^4)).
\end{equation}
Therefore
\begin{equation}\label{a3}
\frac{f(x_{n})}{f'(x_{n})}=\frac{1}{m}e_n-\frac{c_1}{m^2c_0}e_n^2+\frac{-(1+m)c_1^2+2mc_0c_2}{m^3c_0^2}e_n^3+O(e_n^{4}),
\end{equation}
and hence
\begin{equation}\label{a4}
e_{n,y}=y_{n}-\alpha=
\frac{-1+m}{m}e_n-\frac{c_1}{m^2c_0}e_n^2+\frac{-(1+m)c_1^2+2mc_0c_2}{m^3c_0^2}e_n^3+O(e_n^{4}).
\end{equation}
For $f(y_n)$ we also have
\begin{equation}\label{a5}
f(y_{n}) =e_{n,y}^m\left(c_{0} +c_1e_{n,y}+
c_{2}e_{n,y}^{2}+c_{3}e_{n,y}^{3}\right)+O(e_{n,y}^{4}).
\end{equation}
Substituting (\ref{a1})-(\ref{a5}) in (\ref{c2}), we obtain
\begin{equation}
e_{n+1}= D_1e_n+D_2e_n^2+D_3e_n^3+O(e_n^4),
\end{equation}
where
\begin{equation}\label{a6}
D_1=1+\frac{\theta}{-\theta+\left(\frac{-1+m}{m}\right)^mm},
\end{equation}
and
\begin{equation}\label{a7}
D_2=\frac{\theta m^{-2+m}(-m(-1+m)^m+\theta
m^m(-1+m))c_1}{(-1+m)((-1+m)^m-\theta m^m)^2c_0},
\end{equation}
and
\begin{equation}\label{a8}
D_3=\frac{\theta m^{-3+m}A}{2(-1+m)^2((-1+m)^m-\theta
m^m)^3c_0^2},
\end{equation}
where
\begin{equation}\label{a9}
\begin{array}{lrl}
A=(-1+m)^{2m}(-1+m+2m^2)(mc_1^2-2(-1+m)c_0c_2)
\\[1ex]
+2\theta^2(-1+m)^2m^{2m}((1+m)c_1^2-2mc_0c_2)-\theta(-1+m)^{1+m} &
\\[1ex]
(m(3+4m)c_1^2+2(1+m-4m^2)c_0c_2).
\end{array}
\end{equation}
Therefor, to provide the three order of convergence, it is
necessary to choose $D_i=0 \quad (i=1,2)$, so we have
\begin{equation}\label{a8}
\theta=\left(\frac{-1+m}{m}\right)^{-1+m},
\end{equation}
and the error equation becomes
\begin{equation}
e_{n+1}=\left(\frac{mc_1^2-2(-1+m)c_0c_2}{2m^2c_0^2}\right)e_n^3+O(e_n^4),
\end{equation}
and method (\ref{ccc2}) has convergence order three, which proves
the theorem.
\end{proof}

\section{Numerical performance and dynamic behavior}\label{sec:3}
\subsection{Numerical results }
In this section we apply the new method (\ref{ccc2}) to several
benchmark examples and compare them with existing methods that have the same order of convergence.\\
\textbf {The new method } is given by
\begin{equation}
  \begin{array}{lrl}\label{c2}
 y_{n}=x_{n}-\frac{f(x_{n})}{f'(x_{n})},&\\
  [1ex]
x_{n+1}=x_n-\frac{(-1+m)^{-1+m} f(x_n)}{(-1+m)^{-1+m}
f(x_n)-m^{-1+m}f(y_n)}\frac{f(x_n)}{f'(x_n)}.
      \end{array}
\end{equation}
\textbf {The Osada's method }\cite{Osada}, is given by
\begin{equation}\label{m1}
x_{n+1}=x_n-\frac{1}{2}m(m+1)\frac{f(x_n)}{f'(x_n)}+\frac{1}{2}(m-1)^2\frac{f'(x_n)}{f''(x_n)}.
\end{equation}
\textbf {The Dong's method} \cite{Dong}, is given by
\begin{equation}
  \begin{array}{lrl}\label{m2}
 y_n=x_n+\sqrt{m}\frac{f(x_n)}{f'(x_n)},&\\
  [1ex]
x_{n+1}=y_n-m(1-\frac{1}{\sqrt{m}})^{1-m}\frac{f(y_n)}{f'(x_n)}.
  \end{array}
\end{equation}
\textbf {The Chun's method }\cite{chun}, is given by
\begin{equation}
  \begin{array}{lrl}\label{m3}
 x_{n+1}=x_n-\frac{m((2\gamma-1)m+3-2\gamma)}{2}\frac{f(x_n)}{f'(x_n)}&\\
  [1ex]
~~~~~~+\frac{\gamma(m-1)^2}{2}\frac{f'(x_n)}{f''(x_n)}-\frac{(1-\gamma)m^2}{2}\frac{f(x_n)^2f''(x_n)}{f'(x_n)^3}.
  \end{array}
\end{equation}
In the numerical experiments of this paper we will use
$\gamma=-1$.\\
We have tested the method (\ref{c2}) on a number of nonlinear
equations. To obtain a high accuracy and avoid the loss of
significant digits, we employed multi-precision arithmetic with
100 significant decimal digits in the programming package of
Mathematica 8 \cite{Hazrat}.\\
In order to test our proposed method (\ref{c2}) also compare it
with the methods (\ref{m1}), (\ref{m2}) and (\ref{m3}) we compute
the error, the computational order of convergence (COC) by the
approximate formula \cite{coc}
\begin{equation}
\textup{COC}\approx\frac{\ln|(x_{n+1}-\alpha)/(x_{n}-\alpha)|}{\ln|(x_{n}-\alpha)/(x_{n-1}-\alpha)|}.
\end{equation}
And the approximated computational order of convergence, (ACOC) by
the formula \cite{acoc}
\[
\textup{ACOC}\approx\frac{\ln|(x_{n+1}-x_{n})/(x_{n}-x_{n-1})|}{\ln|(x_{n}-x_{n-1})/(x_{n-1}-x_{n-2})|}.
\]\\
\begin{table}[h]
\begin{center}
\vspace*{-3ex} \begin{tabular}{l c c}
  \hline
    Test function $f_n$ & root $\alpha$
  \\ \hline
    $f_1(x) = (\sin^2x+x)^5$ & $0$
  \\[0.6ex]
   $f_2(x) = (\ln x+\sqrt{x}-5)^3$ & $8.3094326942315717953469556827$
  \\[0.6ex]
   $ f_3(x)=(\mathrm{e}^{x^2+7x-30}-1)^6$ & $ 3 $
  \\[0.6ex]
 $ f_4(x)=(\sqrt{x}-\frac{1}{x}-3)^7 $& $ 9.6335955628326951924063127092
 $\\
  \hline
\end{tabular}
\end{center}
\vspace*{-3ex}\caption{Test functions $f_1, \dots, f_4$ and root
$\alpha$.\label{table1}}
\end{table}
\begin{table}[!ht]
\begin{center}
\begin{tabular}{c l l l l}
\hline
 ~~ & Method $(\ref{c2})$ & Method $(\ref{m1})$ & Method $(\ref{m2})$ & Method $(\ref{m3})$\\
\hline
$f_1$, $x_0=0.1$\\
$\vert x_{1}-x^{*} \vert$ & $0.739\mathrm{e}{-}3$ & $0.162\mathrm{e}{-}2$ & $0.955\mathrm{e}{-}3$ & $0.112\mathrm{e}{-}2$  \\
$\vert x_{2}-x^{*} \vert$ & $0.364\mathrm{e}{-}9$& $0.106\mathrm{e}{-}7$ & $0.111\mathrm{e}{-}8$ & $0.215\mathrm{e}{-}8$  \\
$\vert x_{3}-x^{*} \vert$ & $0.434\mathrm{e}{-}28$ & $0.304\mathrm{e}{-}23$ & $0.175\mathrm{e}{-}26$ & $0.149\mathrm{e}{-}25$ \\
COC & $3.0000$ & $3.0000$ & $3.0000$ & $3.0000$\\
ACOC & $2.9999$ & $2.9994$ & $2.9998$ & $2.9998$\\
\\
$f_2$, $x_0=8$\\
$\vert x_{1}-x^{*} \vert$ & $0.166\mathrm{e}{-}4$ & $0.719\mathrm{e}{-}4$ & $0.277\mathrm{e}{-}4$ & $0.376\mathrm{e}{-}4$  \\
$\vert x_{2}-x^{*} \vert$ & $0.246\mathrm{e}{-}17$& $0.876\mathrm{e}{-}15$ & $0.189\mathrm{e}{-}16$ & $0.667\mathrm{e}{-}16$  \\
$\vert x_{3}-x^{*} \vert$ & $0.793\mathrm{e}{-}39$ & $0.793\mathrm{e}{-}39$ & $0.793\mathrm{e}{-}39$ & $0.793\mathrm{e}{-}39$ \\
COC & $3.0000$ & $3.0000$ & $3.0000$ & $3.0000$\\
ACOC & $3.0000$ & $3.0000$ & $3.0000$ & $3.0000$\\
\\
$f_3$, $x_0=3.1$\\
$\vert x_{1}-x^{*} \vert$ & $0.128\mathrm{e}{-}1$ & $0.268\mathrm{e}{-}1$ & $0.212\mathrm{e}{-}1$ & $0.216\mathrm{e}{-}1$  \\
$\vert x_{2}-x^{*} \vert$ & $0.325\mathrm{e}{-}4$& $0.105\mathrm{e}{-}2$ & $0.333\mathrm{e}{-}3$ & $0.362\mathrm{e}{-}3$  \\
$\vert x_{3}-x^{*} \vert$ & $0.531\mathrm{e}{-}12$ & $0.871\mathrm{e}{-}7$ & $0.143\mathrm{e}{-}86$ & $0.190\mathrm{e}{-}8$ \\
COC & $3.0000$ & $2.9964$ & $2.9993$ & $2.9993$\\
ACOC & $3.0004$ & $2.9428$ & $2.9857$ & $2.9858$\\
\\
$f_4$, $x_0=9$\\
$\vert x_{1}-x^{*} \vert$ & $0.223\mathrm{e}{-}3$ & $0.792\mathrm{e}{-}4$ & $0.122\mathrm{e}{-}4$ & $0.987\mathrm{e}{-}4$  \\
$\vert x_{2}-x^{*} \vert$ & $0.864\mathrm{e}{-}14$& $0.139\mathrm{e}{-}15$ & $0.742\mathrm{e}{-}19$ & $0.340\mathrm{e}{-}15$  \\
$\vert x_{3}-x^{*} \vert$ & $0.918\mathrm{e}{-}39$ & $0.918\mathrm{e}{-}39$ & $0.918\mathrm{e}{-}39$ & $0.918\mathrm{e}{-}39$ \\
COC & $3.0000$ & $3.0000$ & $3.0000$ & $3.0000$\\
ACOC & $3.0000$ & $3.0000$ & $3.0000$ & $3.0000$\\
\hline
\end{tabular}
\end{center}
\vspace*{-3ex} \caption{Errors, COC and ACOC for methods
(\ref{c2}), (\ref{m1}), (\ref{m2}) and (\ref{m3}). \label{table2}}
\end{table}
\\
In Table \ref{table2} our new method (\ref{c2}) is compared with
the methods (\ref{m1}),(\ref{m2}) and (\ref{m3}) on four nonlinear
equations which were presented in Table \ref{table1}.
\subsection{Dynamic behavior }
We already observed that all methods converge if the initial guess
is chosen suitably. We now investigate the stability region. In
other words, we numerically approximate the domain of attraction
of the zeros as a qualitative measure of stability. To answer the
important question on the dynamical behavior of the algorithms, we
investigate the dynamics of the new method and compare them with
common and well-performing methods from the literature. For more
details one can consult \cite{Lotfi1,Scott,Sharifi,Varona}.

Let $G:\mathbf{C} \to \mathbf{C} $ be a rational map on the
complex plane. For $z\in \mathbf{C} $, we define its orbit as the
set $orb(z)=\{z,\,G(z),\,G^2(z),\dots\}$. A point $z_0 \in
\mathbf{C} $ is called periodic point with minimal period $m$ if
$G^m(z_0)=z_0$, where $m$ is the smallest integer with this
property. A periodic point with minimal period $1$ is called fixed
point. Moreover, a point $z_0$ is called attracting if
$|G'(z_0)|<1$, repelling if $|G^{\prime}(z_0)|>1$, and neutral
otherwise. The Julia set of a nonlinear map $G(z)$, denoted by
$J(G)$, is the closure of the set of its repelling periodic
points. The complement of $J(G)$ is the Fatou set $F(G)$, where
the basin of attraction of the different roots lie \cite{Babajee}.

For the dynamical point of view, in fact, we take a $512 \times
512$ grid of the square $[-3,3]\times[-3,3]\in \mathbf{C}$ and
assign a color to each point $z_0\in D$ according to the root to
which the corresponding orbit of the iterative method starting
from $z_0$ converges, and we mark the point as black if the orbit
does not converge to a root, in the sense that after at most 100
iterations it has a distance to any of the roots, which is larger
than $10^{-3}$. In this way, we distinguish the attraction basins
by their color for different methods.

We have tested several different examples, and the results on the
performance of the tested methods were similar. Therefore, we
report the general observation here for following test problems
which are presenter in Table \ref{table3}.

\begin{table}[!ht]
\begin{center}
\begin{tabular}{l c}
\hline Test problem  & roots\\
\hline $p_1(z)=(z^3-1)^{10}$&$1,~ -0.5 \pm 0.866025 i$ \\
$p_2(z)=(z^5-z^2+1)^{15}$& $-0.808731,~-0.464912 \pm 1.07147 i,~0.869278\pm0.388269 i$\\
$p_2(z)=(2z^4-z)^8$&$0,~-0.39685 \pm 0.687365 i,~0.793701$ \\
 \hline
\end{tabular}
\end{center}
\vspace*{-3ex} \caption{Test problems $p_1, p_2, p_3$ and their
roots \label{table3}}
\end{table}

\begin{figure}[ht!]
\begin{minipage}[b]{0.22\linewidth}
\centering
\includegraphics[width=\textwidth]{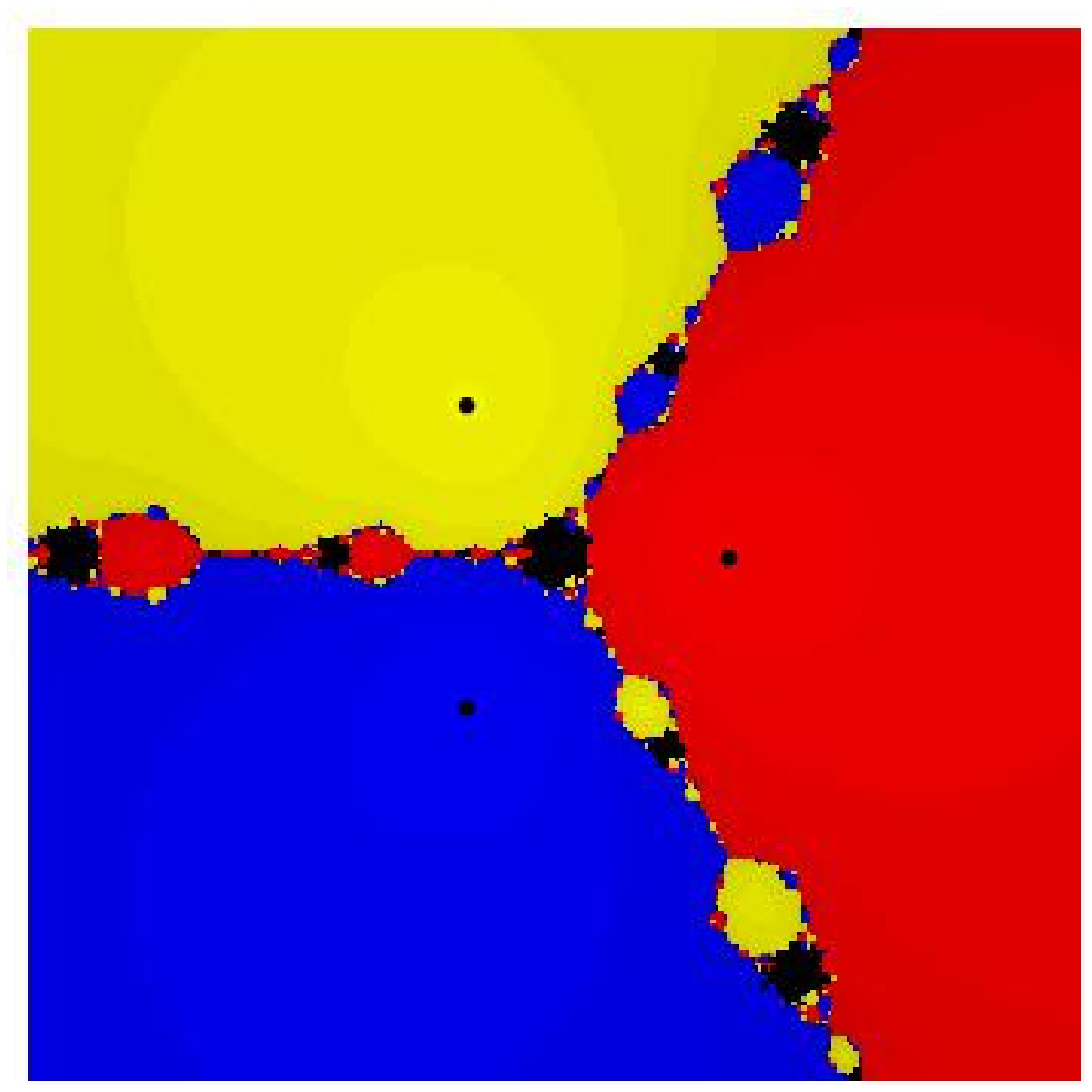}
\end{minipage}
\hspace{0.3cm}
\begin{minipage}[b]{0.22\linewidth}
\centering
\includegraphics[width=\textwidth]{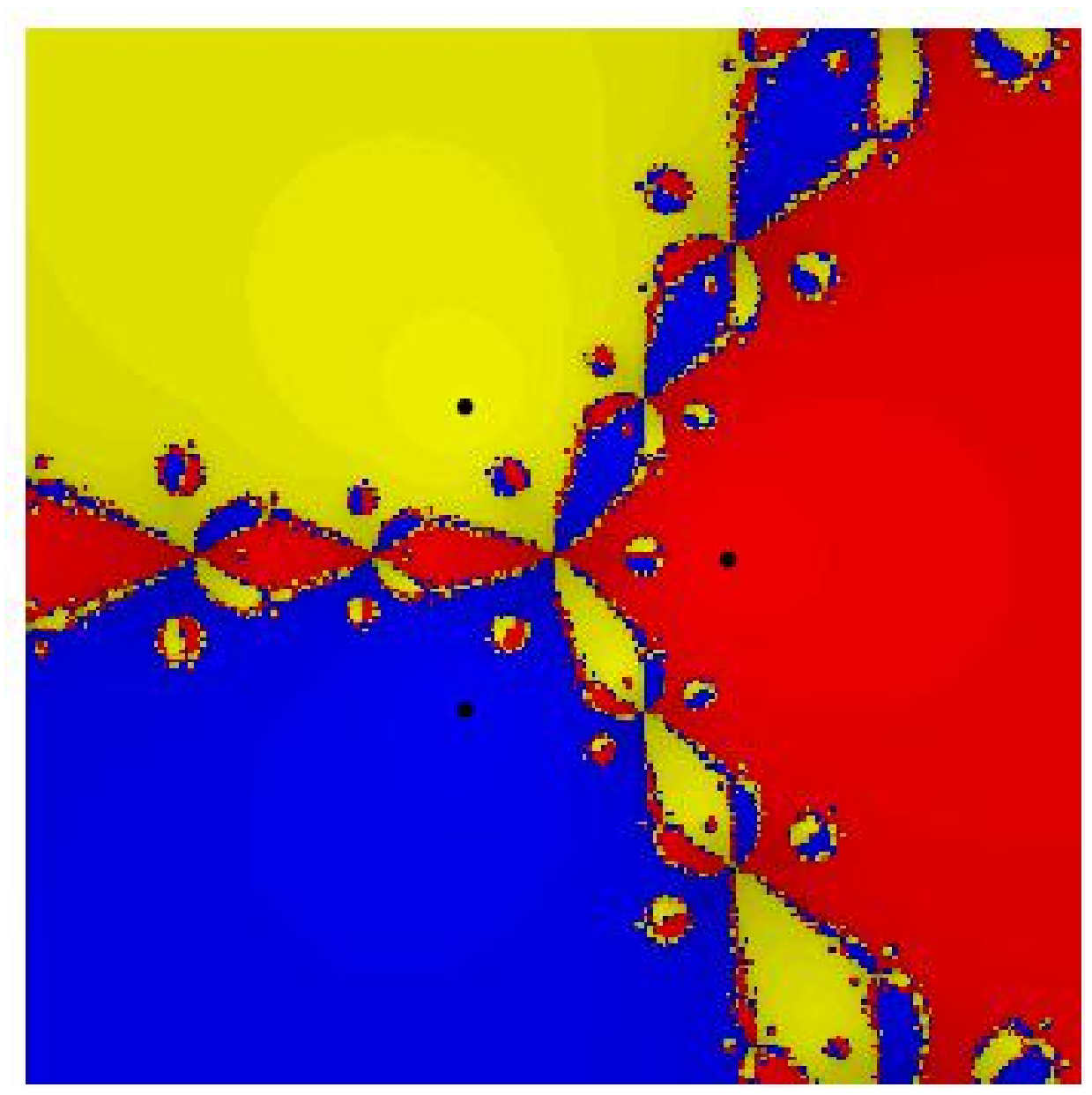}
\end{minipage}
\hspace{0.3cm}
\begin{minipage}[b]{0.22\linewidth}
\centering
\includegraphics[width=\textwidth]{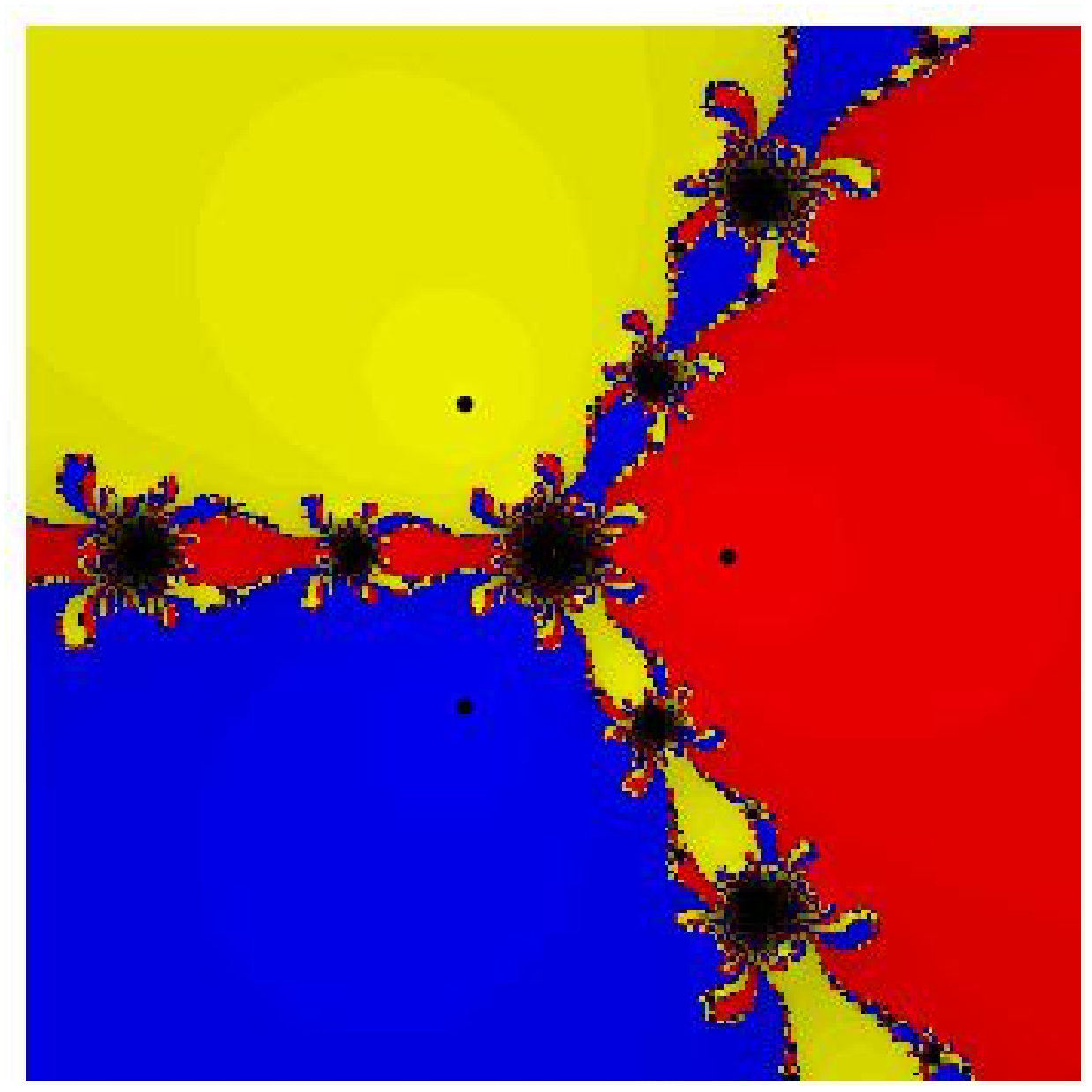}
\end{minipage}
\hspace{0.3cm}
\begin{minipage}[b]{0.22\linewidth}
\centering
\includegraphics[width=\textwidth]{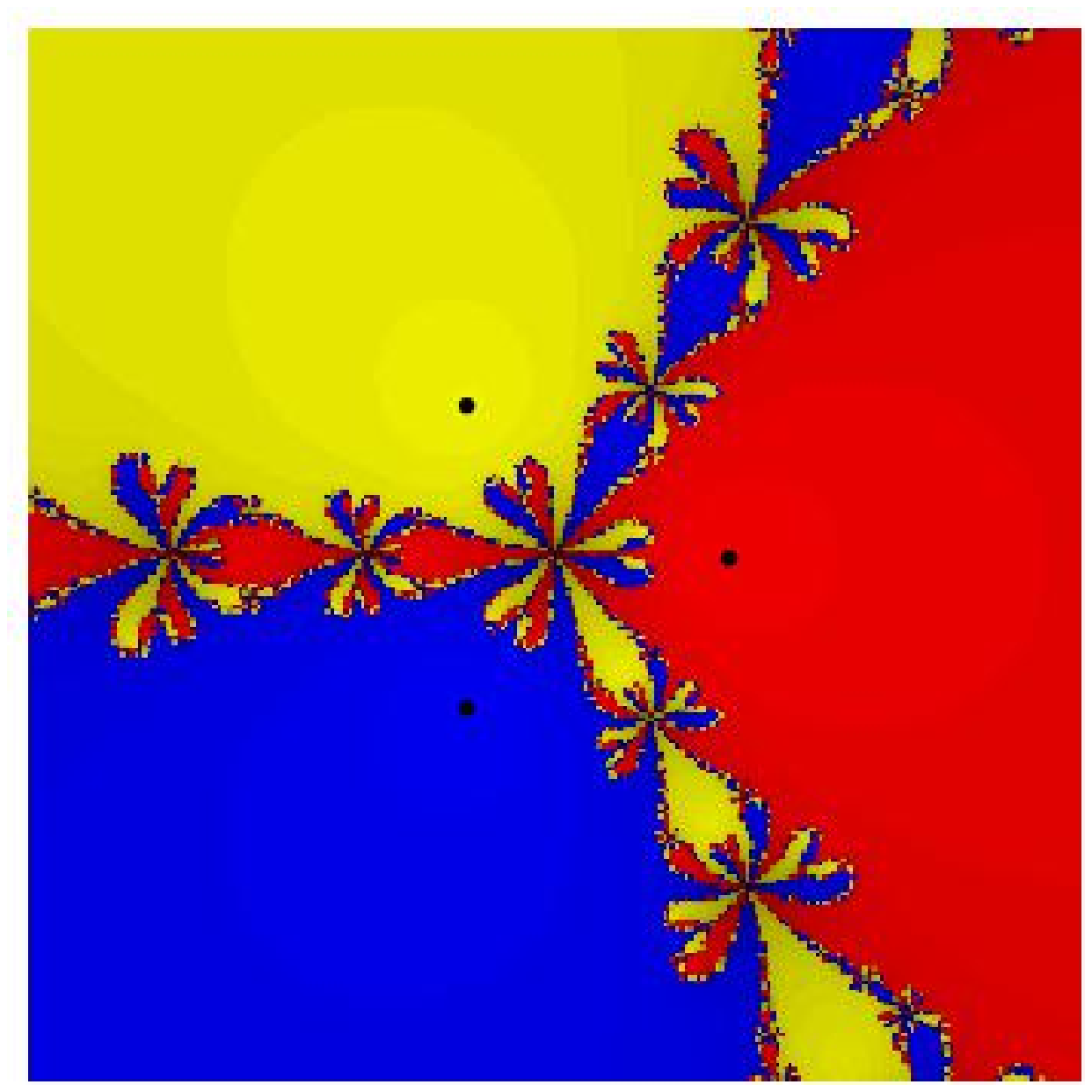}
\end{minipage}
\caption{Comparison of basin of attraction of methods (\ref{c2}),
(\ref{m1}), (\ref{m2}) and (\ref{m3})  for test problem
$p_1(z)=(z^3-1)^{10}$} \label{fig:figure1}
\end{figure}
\begin{figure}[ht!]
\begin{minipage}[b]{0.22\linewidth}
\centering
\includegraphics[width=\textwidth]{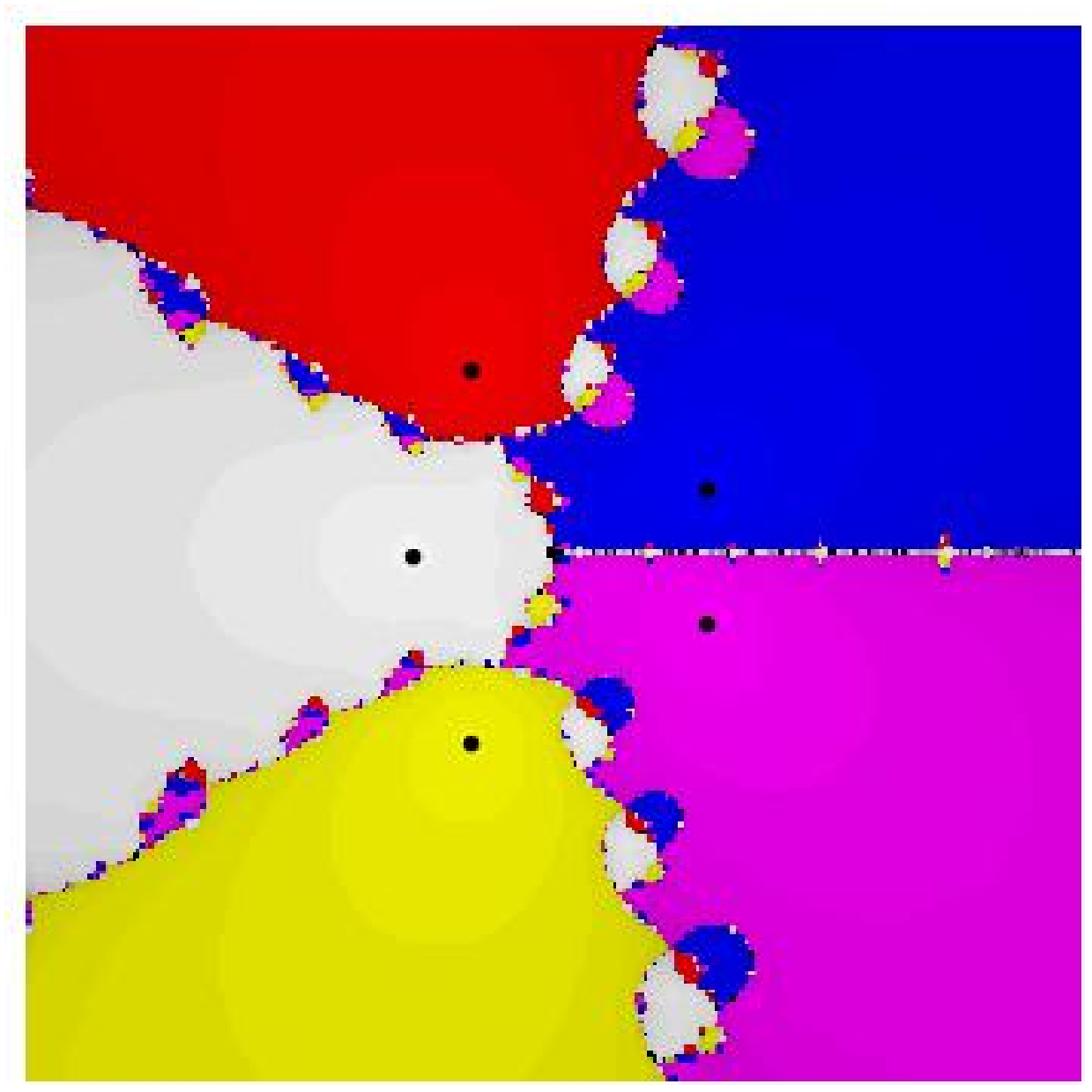}
\end{minipage}
\hspace{0.3cm}
\begin{minipage}[b]{0.22\linewidth}
\centering
\includegraphics[width=\textwidth]{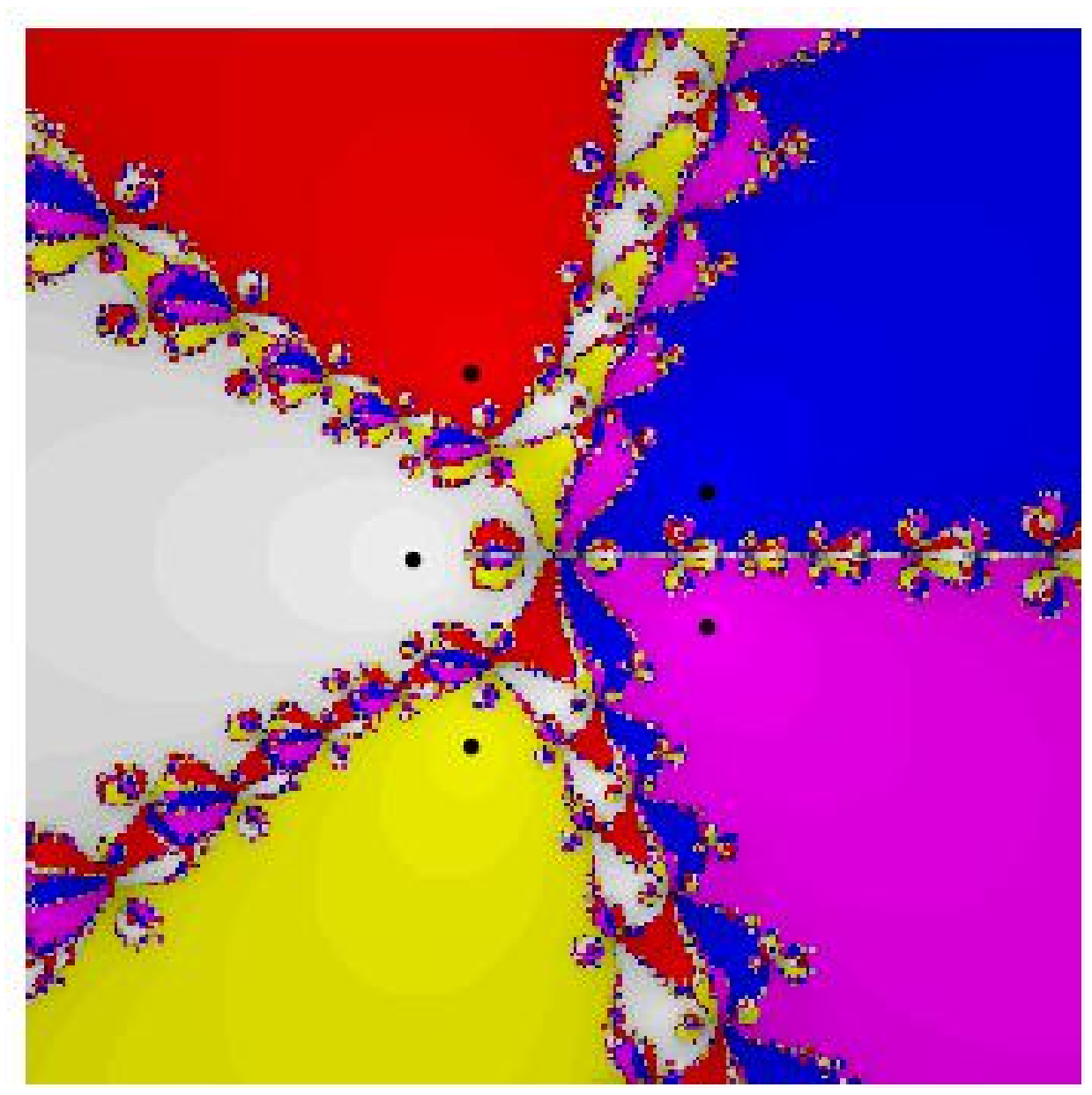}
\end{minipage}
\hspace{0.3cm}
\begin{minipage}[b]{0.22\linewidth}
\centering
\includegraphics[width=\textwidth]{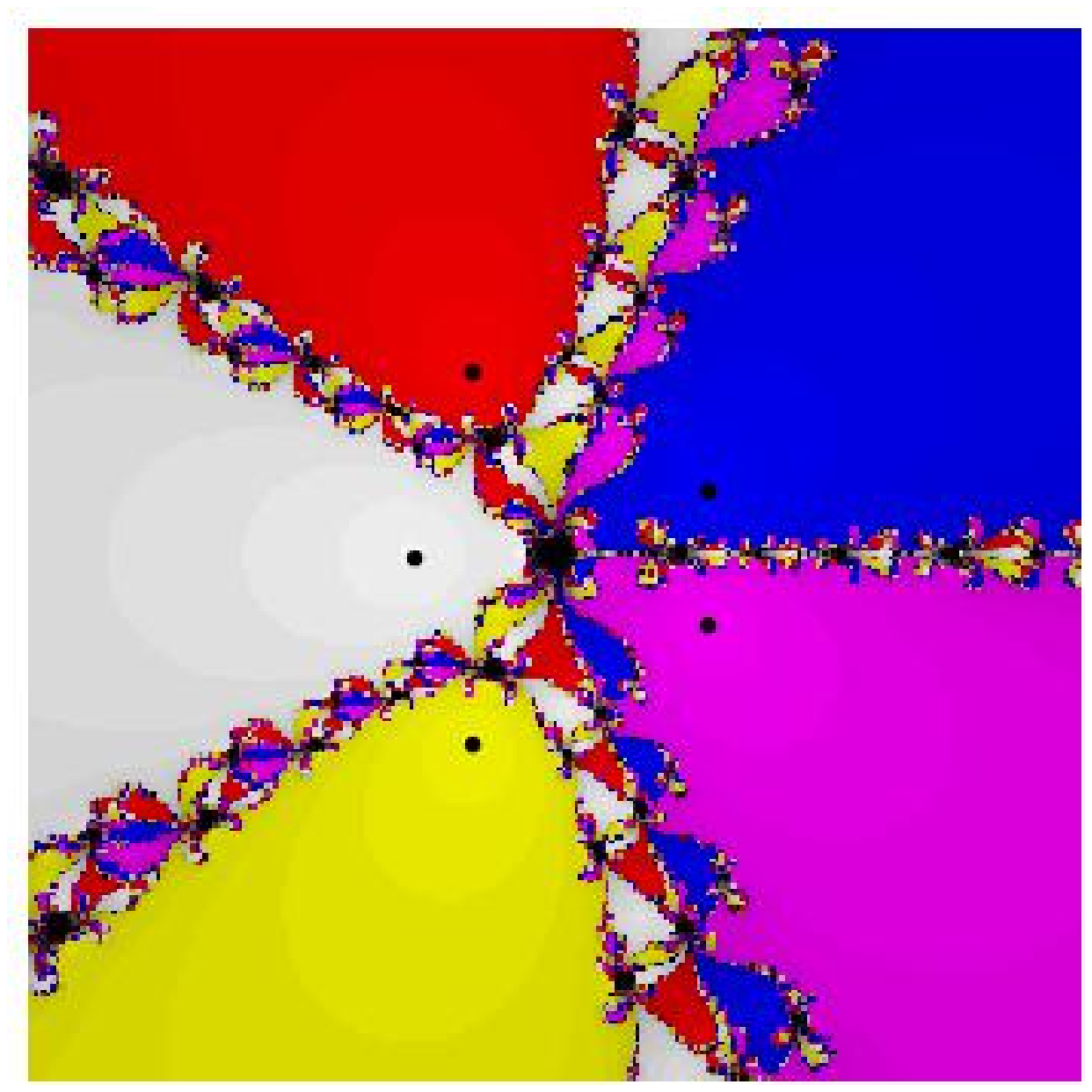}
\end{minipage}
\hspace{0.3cm}
\begin{minipage}[b]{0.22\linewidth}
\centering
\includegraphics[width=\textwidth]{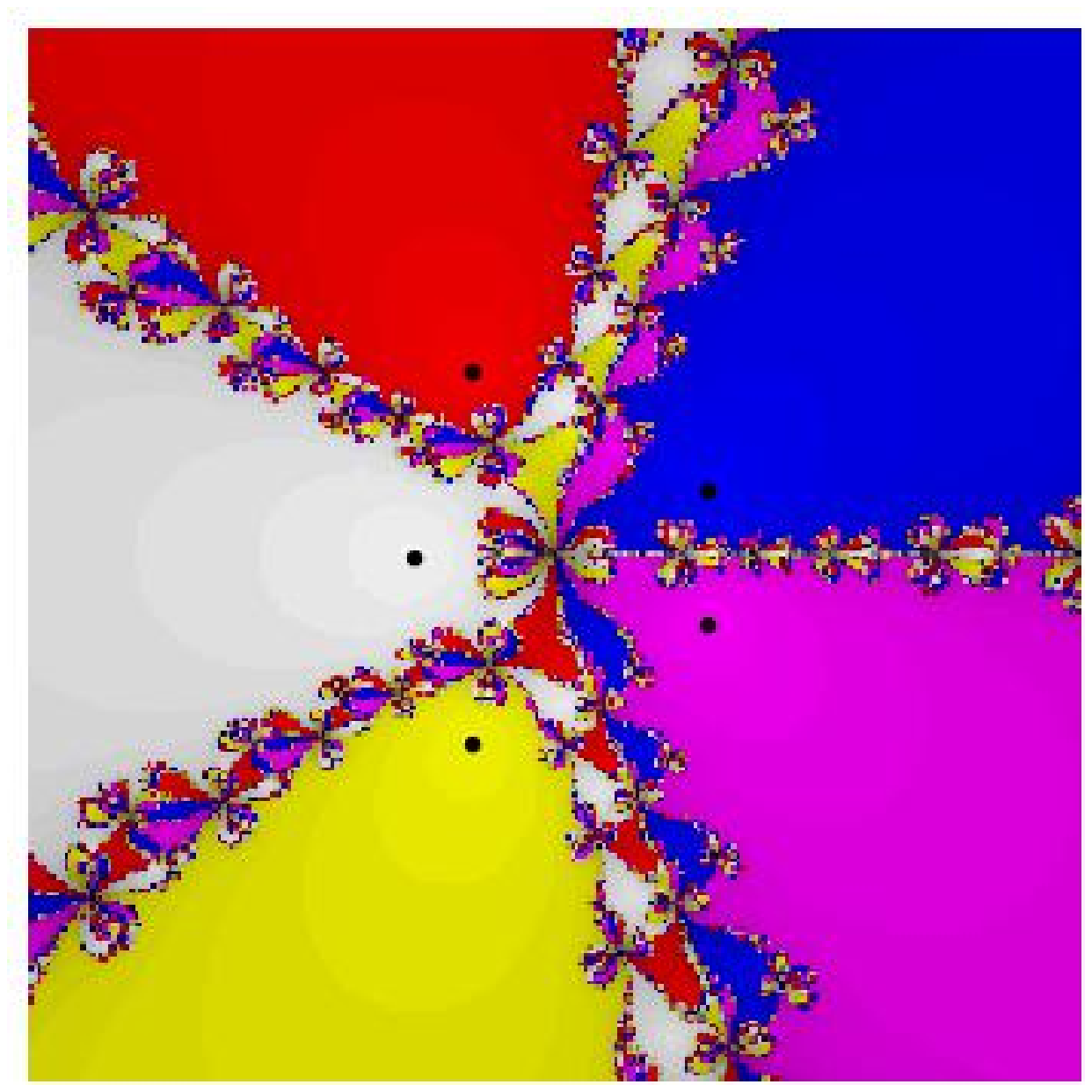}
\end{minipage}
\caption{Comparison of basin of attraction of methods (\ref{c2}),
(\ref{m1}), (\ref{m2}) and (\ref{m3}) for test problem
$p_2(z)=(z^5-z^2+1)^{15}$} \label{fig:figure2}
\end{figure}
\begin{figure}[ht!]
\begin{minipage}[b]{0.22\linewidth}
\centering
\includegraphics[width=\textwidth]{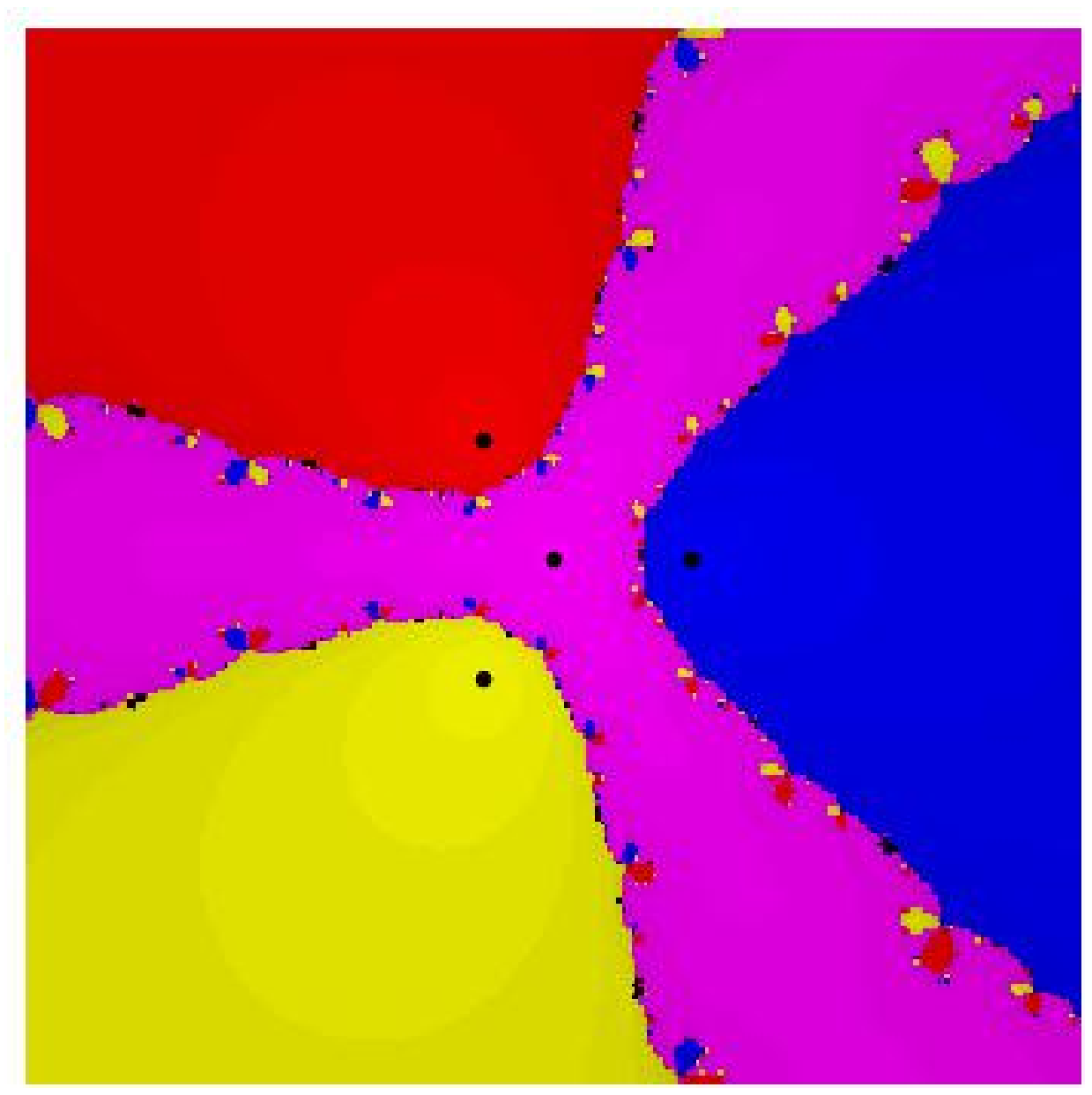}
\end{minipage}
\hspace{0.3cm}
\begin{minipage}[b]{0.22\linewidth}
\centering
\includegraphics[width=\textwidth]{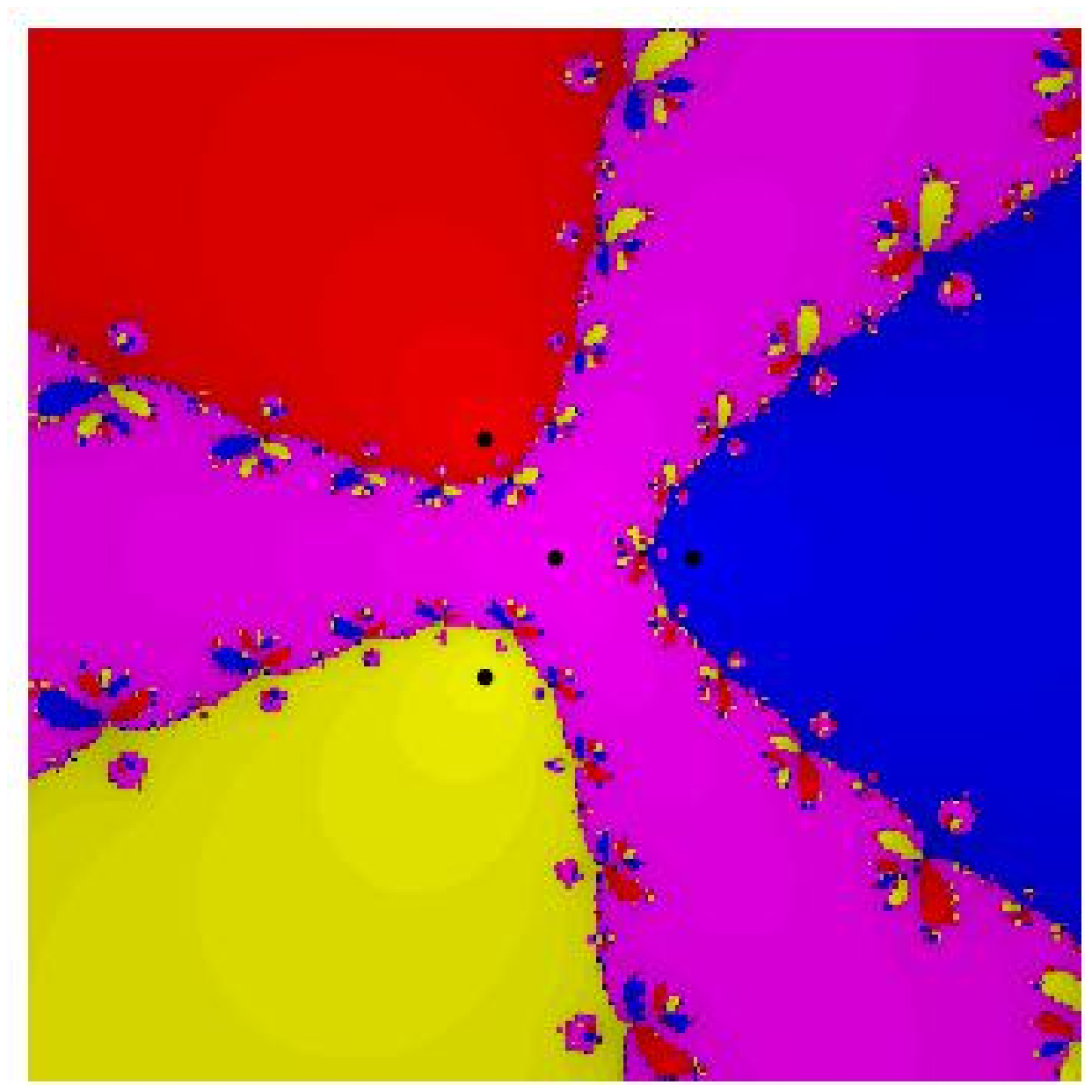}
\end{minipage}
\hspace{0.3cm}
\begin{minipage}[b]{0.22\linewidth}
\centering
\includegraphics[width=\textwidth]{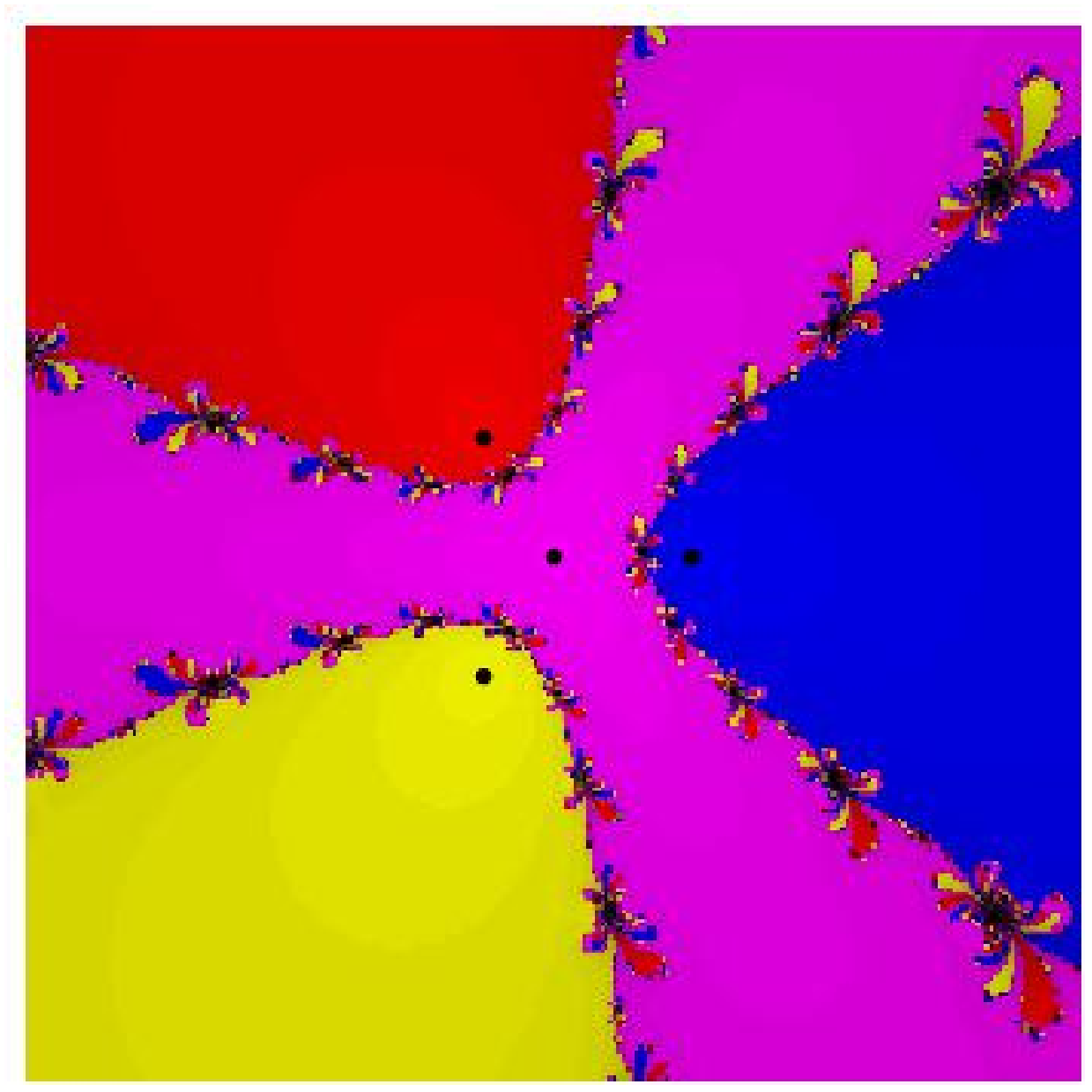}
\end{minipage}
\hspace{0.3cm}
\begin{minipage}[b]{0.22\linewidth}
\centering
\includegraphics[width=\textwidth]{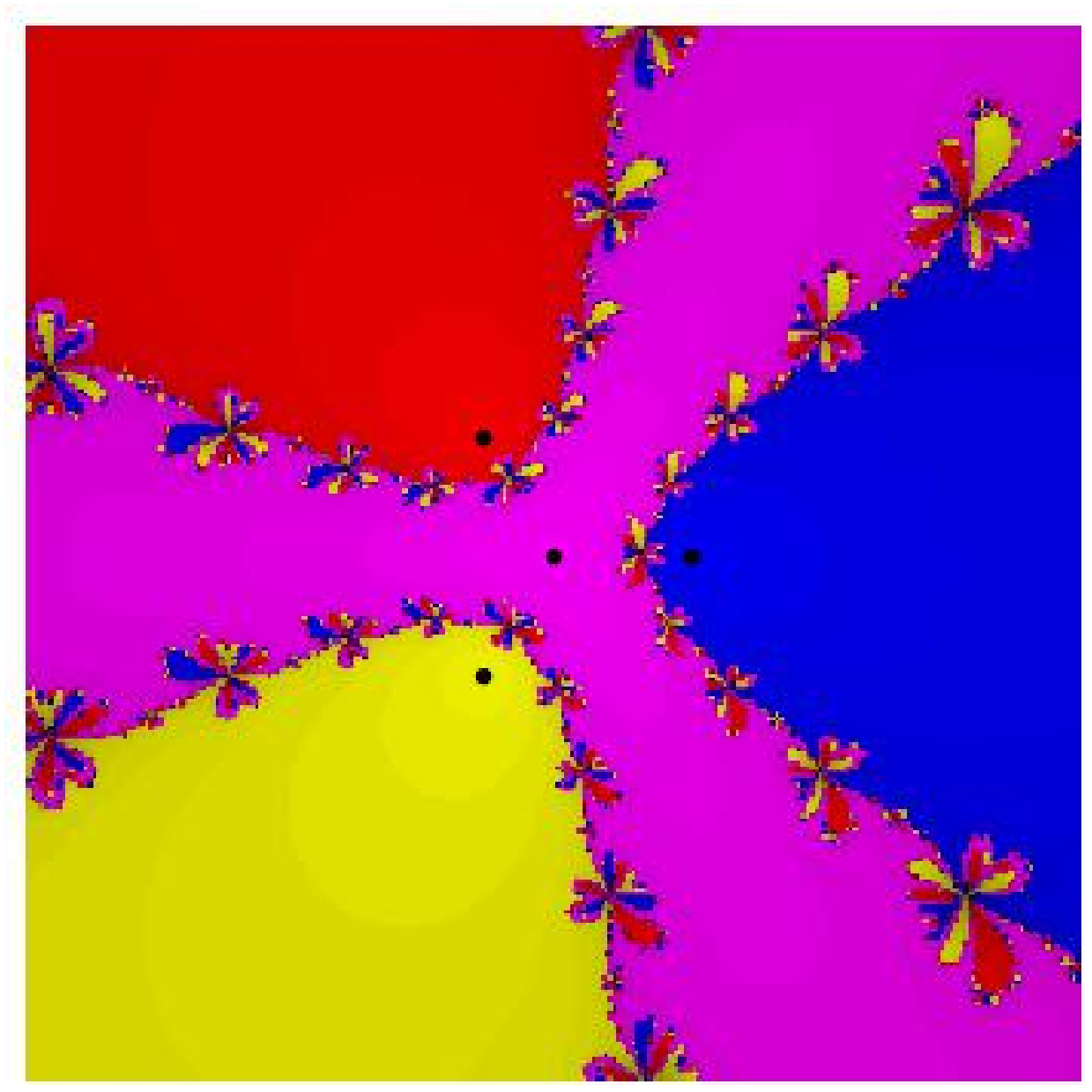}
\end{minipage}
\caption{Comparison of basin of attraction of methods (\ref{c2}),
(\ref{m1}), (\ref{m2}) and (\ref{m3})  for test problem
$p_3(z)=(2z^4-z)^8$} \label{fig:figure3}
\end{figure}
\newpage
In Figures \ref{fig:figure1}, \ref{fig:figure2} and
\ref{fig:figure3}, basins of attractions of methods (\ref{c2}),
(\ref{m1}), (\ref{m2}) and (\ref{m3}) with three test problems are
illustrated from left to right respectively. As a result, in
Figures \ref{fig:figure1}-\ref{fig:figure3} basins of attraction
of method (\ref{c2}) seem to produce larger basins of attraction
than other methods.
\section{Conclusion}\label{sec:4}
In this paper, Newton-Secant's method for simple zeros was
modified for finding multiple zeros of non-linear equations with
same order of convergence and without any additional evaluations
of the function or its derivatives. A numerical comparison with
other methods shows that our new method is a valuable alternative
to existing methods. In addition, a numerical investigation of the
basins of attraction of the solutions illustrated that the
stability region of our method it typically larger than that of
other methods.

\end{document}